\documentclass[11pt]{amsart}

\usepackage{amsmath}
\usepackage{graphicx}

\makeatletter
\renewcommand\subsection{\@startsection{subsection}{2}{0mm}
{-10.5dd plus-8pt minus-4pt}{10.5dd}
{\normalsize\upshape}}
\makeatother

\renewcommand{\index}[1]{ }
\theoremstyle{plain}
\newtheorem{theorem}{Theorem}[section]
\newtheorem{lemma}[theorem]{Lemma}
\newtheorem{proposition}[theorem]{Proposition}
\newtheorem{corollary}[theorem]{Corollary}

\theoremstyle{definition}
\newtheorem{remark}[theorem]{Remark}
\newtheorem{definition}[theorem]{Definition}
\newtheorem{example}[theorem]{Example}
\newtheorem{question}[theorem]{Question}

\newtheorem{claim}{Claim}

\newcommand{\figer}[4]{
\begin{figure}[h]
\begin{center}
\includegraphics[width=#2\textwidth]{#1.eps}
\end{center}
\caption{#4 }
\label{#3}
\end{figure}}

\numberwithin{figure}{section}

\newcommand{\ZZ}{\mathbb{Z}}
\newcommand{\interior}{\operatorname{int}}
\newcommand{\Fix}{\operatorname{Fix}}
\newcommand{\cl}{\operatorname{cl}}
\newcommand{\fr}{\operatorname{fr}}
\newcommand{\Isom}{\operatorname{Isom}}
\newcommand{\lk}{\operatorname{lk}}
\newcommand{\Sym}{\operatorname{Sym}}
\newcommand{\Diff}{\operatorname{Diff}}
\newcommand{\Area}{\operatorname{Area}}
\newcommand{\cob}{\mathcal{C}}
\newcommand{\FF}{\mathcal{F}}

\begin{document}

\title[Invariant Seifert surfaces]
{Invariant Seifert surfaces for strongly invertible knots}

\author{Mikami Hirasawa}
\address{Department of Mathematics\\
Nagoya Institute of Technology\\ 
Showa-ku, Nagoya city, Aichi, 466-8555, Japan}
\email{hirasawa.mikami@nitech.ac.jp}

\author{Ryota Hiura}
\address{Inuyama-Minami High School\\
Hasuike 2-21, Inuyama City, Aichi, 484-0835, Japan}
\email{ryhiura@gmail.com}

\author{Makoto Sakuma}
\address{Osaka Central Advanced Mathematical Institute\\
Osaka Metropolitan University\\
3-3-138, Sugimoto, Sumiyoshi, Osaka City
558-8585, Japan}
\email{sakuma@hiroshima-u.ac.jp}

\makeatletter
\@namedef{subjclassname@2020}{%
\textup{2020} Mathematics Subject Classification}
\makeatother

\keywords{strongly invertible knot, invariant Seifert surface, equivariant genus, Kakimizu complex\\   %\\
{\bf This article is to appear in the book,
Essays in geometry, dedicated to Norbert \mbox{A'Campo} (ed. A. Papadopoulos), European Mathematical Society Press, Berlin,  2023.}}
%\subjclass{Primary 57K10; secondary 57M60}
%57K10 Knot theory
%57M60 Group actions on manifolds and cell complexes in low dimensions

\begin{abstract}
Abstract: We study invariant Seifert surfaces for strongly invertible knots, and prove
that the gap between the equivariant genus 
(the minimum of the genera of invariant Seifert surfaces)
of a strongly invertible knot and the (usual) genus of the underlying knot 
can be arbitrarily large.  
This forms a sharp contrast with Edmonds' theorem that
every periodic knot admits an invariant minimal genus Seifert surface. 
We also prove variants of Edmonds' theorem, 
which are useful in studying
invariant Seifert surfaces for strongly invertible knots.

\end{abstract}

\maketitle

\vspace*{-5mm}
\begin{center}
{\it Dedicated to Professor Norbert A'Campo on his 80th birthday}
\end{center}

\section{Introduction} \label{sec:intro}
A smooth knot $K$ in $S^3$ is said to be 
({\it cyclically}) {\it periodic with period $n$}\index{knot!periodic knot} %% 
if there is a periodic diffeomorphism $f$ of $S^3$ 
of period $n$
which leaves $K$ invariant and fixes a simple loop in the knot complement $S^3\setminus K$.
Since the seminal work by Trotter \cite{Trotter1961} and Murasugi \cite{Murasugi1971},
periodic knots have been studied extensively.
In particular, Edmonds and Livingston \cite{Edmonds-Livingston}
proved that every periodic knot admits an invariant incompressible Seifert surface,
and this was enhanced by
Edmonds \cite{Edmonds} to the existence of an invariant minimal genus Seifert surface.
(He applied this result to prove Fox's conjecture 
that a given nontrivial knot has only finitely many periods.)

It is natural to ask if the same result holds for strongly invertible knots.
Recall that a smooth knot $K$ in $S^3$ is said to be 
{\it strongly invertible}
if there is a smooth involution $h$ of $S^3$
which leaves $K$ invariant and fixes a simple loop intersecting $K$ in two points.
The involution $h$ is called a 
{\it strong inversion}\index{strong inversion} %% 
of $K$.
As in \cite{Sakuma1986},
we use the term 
{\it strongly invertible knot}\index{knot!strongly invertible knot} %%
to mean a pair $(K,h)$ of a knot $K$ and a strong inversion $h$ of $K$,
and regard two strongly invertible knots $(K,h)$ and $(K',h')$ 
to be {\it equivalent}
if there is an orientation-preserving diffeomorphism $\varphi$
of $S^3$ mapping $K$ to $K'$ such that $h'=\varphi h\varphi^{-1}$.

Note that if $S$ is an $h$-invariant Seifert surface for a strongly invertible knot $(K,h)$
then $\Fix(h)\cap S$ is equal to one of the two subarcs of 
$\Fix(h)\cong S^1$
bounded by  
$\Fix(h)\cap K\cong S^0$.
So the problem of whether $K$ admits an invariant minimal genus Seifert surface
depends on the choice of the subarc of $\Fix(h)$, in addition to the choice of 
the strong inversion $h$.
By a {\it marked strongly invertible knot}\index{knot!marked strongly invertible knot}, %%
we mean a triple $(K,h,\delta)$ where $(K,h)$ is a strongly invertible knot
and $\delta$ is a subarc of $\Fix(h)$ bounded by $\Fix(h)\cap K$.
Two marked strongly invertible knots $(K,h,\delta)$ and $(K',h',\delta')$ are regarded to be
{\it equivalent} if 
there is an orientation-preserving diffeomorphism $\varphi$
of $S^3$ mapping $K$ to $K'$ such that 
$h'=\varphi h\varphi^{-1}$ and 
$\delta'=\varphi(\delta)$.

\begin{definition}
    \label{def:equivariant-genus}
    {\rm
    By an {\it invariant Seifert surface}\index{surface!invariant Seifert surface}
    \index{Seifert surface!invariant Seifert surface} %%
     (or a {\it Seifert surface} in brief)
    {\it for a marked strongly invertible knot $(K,h,\delta)$}, 
    we mean a Seifert surface $S$ for $K$ 
    such that $h(S)=S$ and $\Fix(h)\cap S=\delta$.
    The {\it equivariant genus}\index{genus!equivariant genus} %%
     (or the {\it genus} in brief) $g(K,h,\delta)$ of $(K,h,\delta)$ 
    is defined to be the minimum
    of the genera of Seifert surfaces for $(K,h,\delta)$.
    A Seifert surface for $(K,h,\delta)$ is said to be of 
    {\it minimal genus}\index{genus!minimal genus} %%
    if its genus is equal to $g(K,h,\delta)$.
    }
    \end{definition}

Every marked strongly invertible knot admits an invariant Seifert surface 
(Proposition \ref{prop:existence-invariant-Seifert}), and so 
its equivariant genus is well-defined.
However, in general, 
the equivariant genus is bigger than the (usual) genus.
In fact, the following theorem is proved in the second author's 
master thesis \cite{Hiura} supervised by the third author with support by the first author.

\begin{theorem}
\label{thm:main}
For any integer $n\ge 0$, 
there exits a marked strongly invertible knot $(K,h,\delta)$
such that $g(K,h,\delta)-g(K)=n$. 
\end{theorem}

This theorem follows from a formula 
of the equivariant genera for
certain marked strongly invertible knots 
that arise from $2$-bridge knots.
See Examples \ref{example:8-3} and \ref{example:main-example} for 
special simple cases.
However, there remained various marked strongly invertible $2$-bridge knots 
whose equivariant genera were undetermined.

\medskip

This paper and its sequel \cite{Hirasawa-Hiura-Sakuma2} 
are motivated by the desire 
to determine the equivariant genera of 
all marked strongly invertible $2$-bridge knots.

 In this paper, we prove the following two variants of Edmonds' theorem on periodic knots,
 which are useful in 
 studying invariant Seifert surfaces for general strongly invertible knots:
Theorem \ref{thm:disjoint-Seifert2} is 
used in \cite{Hirasawa-Hiura-Sakuma2} to give a unified
 determination of the equivariant genera of 
all marked strongly invertible $2$-bridge knots.
  
\begin{theorem}
\label{thm:disjoint-Seifert}
Let $(K,h)$ be a strongly invertible knot.
Then there is a minimal genus Seifert surface $F$ for $K$ 
such that $F$ and $h(F)$ have disjoint interiors.
\end{theorem}

\begin{theorem}
\label{thm:disjoint-Seifert2}
Let $(K,h,\delta)$ be a marked strongly invertible knot, and 
let $F$ be a minimal genus Seifert surfaces for $K$ 
such that $F$ and $h(F)$ have disjoint interiors.
Then there is a minimal genus Seifert surface $S$ for $(K,h,\delta)$
whose interior is disjoint from the interiors of $F$ and $h(F)$.
\end{theorem}

Note that a Seifert surface for a fibered knot\index{knot!fibered knot} %%
is a fiber surface\index{surface!fiber surface} %% 
if and only if it is of minimal genus.
For fibered knots, we have a stronger conclusion
as in Proposition \ref{prop:fibered} below.

\begin{proposition}
\label{prop:fibered}
Let $K$ be a strongly invertible fibered knot.
Then for any marked strongly invertible knot $(K,h,\delta)$,
there is an $h$-invariant fiber surface for $K$ containing $\delta$,
and hence
$g(K,h,\delta) = g(K)$.
\end{proposition}

\medskip

Theorem \ref{thm:disjoint-Seifert} 
is also motivated by our interest in the 
{\it Kakimizu complex}\index{Kakimizu complex}
$MS(K)$ of the knot $K$
and in the natural action of
the symmetry group\index{symmetry group}
$\Sym(S^3,K)=\pi_0\Diff(S^3,K)$ on $MS(K)$.
The complex was introduced by Kakimizu \cite{Kakimizu1988}
as the flag simplicial complex 
whose vertices correspond to the (isotopy classes of) minimal genus Seifert surfaces for $K$
and edges to pairs of such surfaces with disjoint interiors.
The following corollary of Theorem \ref{thm:disjoint-Seifert} 
may be regarded as a refinement of a special case of 
a theorem proved by Przytycki and Schultens \cite[Theorem 1.2]{Przytycki-Schultens}.

\begin{corollary}
\label{cor:disjoint-Seifert}
Let $(K,h)$ be a strongly invertible knot,
and let $h_*$ be the automorphism of $MS(K)$
induced from of the strong inversion $h$.
Then one of the following holds.
\begin{enumerate}
\item
There exists a vertex of $MS(L)$
that is fixed by $h_*$.
\item
There exists a pair of adjacent vertices of $MS(K)$
that are exchanged by $h_*$. 
(Thus 
the mid point of the edge spanned by the vertices is
fixed by $h_*$.)
\end{enumerate}
\end{corollary}

For a $2$-bridge knot $K$,
the structure of $MS(K)$ is described by \cite[Theorem 3.3]{Sakuma1994};
in particular, the underlying space $|MS(K)|$ is 
identified with a linear quotient of a cube.
The actions of strong inversions
on $MS(K)$ will be described in \cite{Hirasawa-Hiura-Sakuma2}.

\medskip

This paper is organized as follows.
In Section \ref{sec:basic-fact}, 
we recall basic facts about
strongly invertible knots.
In Section \ref{sec:invariant-Seifert-surface},
we give two constructions of invariant Seifert surfaces 
and prove a basic proposition (Proposition \ref{prop:lifting})
concerning the quotient of an invariant Seifert surface
by the strong inversion.
In Section \ref{sec:Proof-MainTheorem},
we give an argument for determining the equivariant genera  
(Proposition \ref{prop:band-number} and Corollary \ref{cor:band-number}),
and prove Theorem \ref{thm:main}
by using that argument (Example \ref{example:main-example}).
In Section \ref{sec:proof-SubMainTheorem},
we prove our main theorems, Theorems 
\ref{thm:disjoint-Seifert} and
\ref{thm:disjoint-Seifert2}.
In the final section, 
Section \ref{sec:fourgenus}, 
we briefly review old and new studies of
equivariant $4$-genera of symmetric knots.

Finally, we explain relation of this paper and A'Campo's work.
Motivated by isolated singularities of complex hypersurfaces,
A'Campo \cite{A'Campo}
formulated a way to construct fibered links in the 3-sphere from 
{\it divides}\index{divide}, 
i.e., immersions of copies of $1$-manifolds in a disk. 
He proved that the link obtained from a connected divide is
strongly invertible and fibered.
In \cite[Section 3]{A'Campo}, we can find a beautiful description of a pair of
invariant fiber surfaces for the link.
This nicely illustrates Proposition \ref{prop:fibered}.
(See \cite{Hirasawa} for visualization of these links and their fiber surfaces.)
Furthermore, Couture \cite{Couture} introduced the more general notion of
ordered Morse signed divides, and proved that
every strongly invertible link is isotopic to the link of an ordered Morse signed divide.
As A'Campo suggested to the authors, it would be interesting to study invariant Seifert surfaces 
from this view point.

\section{Basic facts concerning strongly invertible knots}
\label{sec:basic-fact}

Recall that a knot $K$ in $S^3$ is said
to be {\it invertible}\index{knot!invertible knot} %%
if there is an orientation-preserving diffeomorphism $h$ that maps $K$ to itself 
reversing orientation.
The existence of non-invertible knots was proved by Trotter \cite{Trotter1963}
using $2$-dimensional hyperbolic geometry.
(His proof is based on the fact that the pretzel knots admit the structure
of Seifert fibered orbifolds, where the base orbifolds are generically hyperbolic.)
If $h$ can be chosen to be an involution
then $K$ is said to be 
{\it strongly invertible},\index{strongly invertible}\index{knot!strongly invertible knot} 
and $h$ is called a {\it strong inversion}\index{strong inversion}.  %%
Though strong invertiblity of course implies invertiblity, 
the converse does not hold,
as shown by Whitten \cite{Whitten}.
However, for hyperbolic knots, the converse also holds
by the Mostow-Prasad rigidity theorem.
Moreover, a sufficient condition
for invertible knots to be strongly invertible was given by Boileau \cite{Boileau}.
The finiteness theorem of symmetries of $3$-manifolds proved by Kojima \cite{Kojima} 
by using the orbifold theorem \cite{BLP, BoP, CHK, DL}
implies that  any knot 
admits only finitely many strong inversions up to equivalence.
Here two strong inversions $h$ and $h'$ of $K$ are regarded to be {\it equivalent}
if there is an orientation-preserving diffeomorphism $\varphi$
of $S^3$ such that 
$h'=\varphi h\varphi^{-1}$.
However, the number of strong inversions up to equivalence 
for a satellite knot can be arbitrary large 
\cite[Lemma 5.4]{Sakuma1986b}.
For torus knots and hyperbolic knots,
the number is at most $2$, as described in 
Proposition \ref{prop:strong-inversion} below.
Recall that a knot $K$ in $S^3$ is said to have
{\it cyclic period}\index{cyclic period}\index{period!cyclic period} $2$ %%
or {\it free period}  $2$,\index{free period}\index{period!free period} %%
respectively,
if there is an orientation-preserving smooth involution $f$ of $S^3$ that maps $K$ to itself 
preserving orientation, such that $\Fix(f)$ is $S^1$ or $\emptyset$.

\begin{proposition}
\label{prop:strong-inversion}
{\rm (1)} The trivial knot admits a unique strong inversion up to equivalence.

{\rm (2)} A nontrivial torus knot admits a unique strong inversion up to equivalence.

{\rm (3)} An invertible  hyperbolic knot admits exactly two or one strong inversions up to equivalence
according to whether it has (cyclic or free) period $2$ or not.
\end{proposition}

The first assertion is due to Marumoto \cite[Proposition 2]{Marumoto}, and 
the remaining assertions are proved in \cite[Proposition 3.1]{Sakuma1986}
(cf. \cite[Section 4]{ALSS})
by using the result of Meeks-Scott \cite{Meeks-Scott}
on finite group actions on Seifert fibered spaces
and the orbifold theorem.
Another key ingredient of the proof of the third assertion is
the following consequence of 
Riley's observation \cite[p.124]{Riley} based on the positive solution of the Smith conjecture
\cite{Morgan-Bass}.
For a hyperbolic invertible knot $K$, the orientation-preserving isometry group
of the hyperbolic manifold $S^3\setminus K$ is the dihedral group $D_{2n}$
of order $2n$ for some $n\ge 1$, i.e., 
\begin{align*}
\label{dihdral-group}
\Isom^+(S^3\setminus K) \cong
\langle f, h \ | \ f^n=1,\ h^2=1, hfh^{-1}=f^{-1}\rangle.
\end{align*}
Here $h$ extends to a strong inversion of $K$ and $f$ 
extends to a periodic map of $S^3$ of period $n$ which maps $K$ to itself preserving orientation.
If $n$ is odd, then $K$ does not have cyclic nor free period $2$,
and any strong inversion of $K$ is equivalent to that obtained from $h$.
If $n=2m$ is even, then $f^m$ extends to an involution of $S^3$
which gives cyclic or free period $2$ of $K$,
and any strong inversion of $K$ is equivalent to that obtained from
exactly one of $h$ and $fh$. 

By \cite[Proposition 1.2]{Kodama-Sakuma},
the above Proposition \ref{prop:strong-inversion}(3) is refined to the following proposition
concerning marked strongly invertible knots.

\begin{proposition}
\label{prop:marked-strong-inversion}
Let $K$ be an invertible hyperbolic knot.

{\rm (1)} Suppose $K$ does not have cyclic nor free period $2$.
Then $K$ admits a unique strong inversion up to equivalence,
and the two marked strongly invertible knots associated with $K$ are inequivalent.
Thus there are precisely two marked strongly invertible knots
associated with $K$.

{\rm (2)} Suppose $K$ has cyclic period $2$.
Then $K$ admits precisely two strong inversions up to equivalence,
and for each strong inversion,
the associated two marked strongly invertible knots are inequivalent.
Thus there are precisely four marked strongly invertible knots
associated with $K$.

{\rm (3)} Suppose $K$ has free period $2$.
Then $K$ admits precisely two strong inversions up to equivalence,
and for each strong inversion,
the associated two marked strongly invertible knots are equivalent.
Thus there are precisely two marked strongly invertible knots
associated with $K$.
\end{proposition}

Since every $2$-bridge knot admits cyclic period $2$,
we obtain the following corollary.

\begin{corollary}
\label{cor:nonhyperbolic-msi-2br-knot}
Every hyperbolic $2$-bridge knot has precisely four associated 
marked strongly invertible knots up to equivalence.
\end{corollary}

The four marked strongly invertible knots
associated with a hyperbolic $2$-bridge knot
are (implicitly) presented in \cite[Proposition 3.6]{Sakuma1986} and 
\cite[Section 4]{ALSS}.

\begin{remark}
\label{rem:nonhyperbolic-msi-2br-knot}
{\rm
By Proposition \ref{prop:strong-inversion}(1) and (2), 
we can easily see that the trivial knot has a unique associated 
marked strongly invertible knot and that
every torus knot has precisely two associated 
marked strongly invertible knots.
}
\end{remark}

We note that
Barbensi, Buck, Harrington and Lackenby \cite{Barbensi-Buck-Harrington-Lackenby}
shed new light on the strongly invertible knots
in relation with the knotoids\index{knotoid}
introduced by Turaev \cite{Turaev}.
They prove that there is a $1-1$ correspondence
between unoriented knotoids, up to \lq\lq rotation'',
and strongly invertible knots, 
up to \lq\lq inversion'' \cite[Theorem 1.1]{Barbensi-Buck-Harrington-Lackenby}.
Proposition \ref{prop:marked-strong-inversion} is a variant of 
their result \cite[Theorem 1.3]{Barbensi-Buck-Harrington-Lackenby}
concerning knotoids.

\section{Invertible diagrams and invariant Seifert surfaces}
\label{sec:invariant-Seifert-surface}

In this section, we describe two proofs
of the following basic proposition
which shows the existence of an invariant Seifert surface
for every marked strongly invertible knot.
One is due to Boyle and Issa \cite{Boyle-Issa2021a},
and the other is due to Hiura \cite{Hiura}.

\begin{proposition}
    \label{prop:existence-invariant-Seifert}
    Every marked strongly invertible knot admits an invariant Seifert surface.
    Namely, for every marked strongly invertible knot $(K,h,\delta)$,
    there is an $h$-invariant Seifert surface $S$ for $K$ such that 
    $\Fix(h)\cap S=\delta$.
    \end{proposition}
    
Following \cite[Definition 3.3]{Boyle-Issa2021a}
we say that a symmetric diagram representing a strongly invertible knot $(K,h)$
is 
\begin{enumerate}
\item
{\it intravergent}\index{knot diagram!intravergent diagram}  %%
if $h$ acts as half-rotation around an axis perpendicular to the plane of the diagram
(see Figure \ref{fig:F1}(a)), and
\item
{\it transvergent}\index{knot diagram!transvergent diagram} %%
if $h$ acts as half-rotation around an axis contained within the plane of the diagram
(see Figure \ref{fig:F5}(a)).
\end{enumerate}

\figer{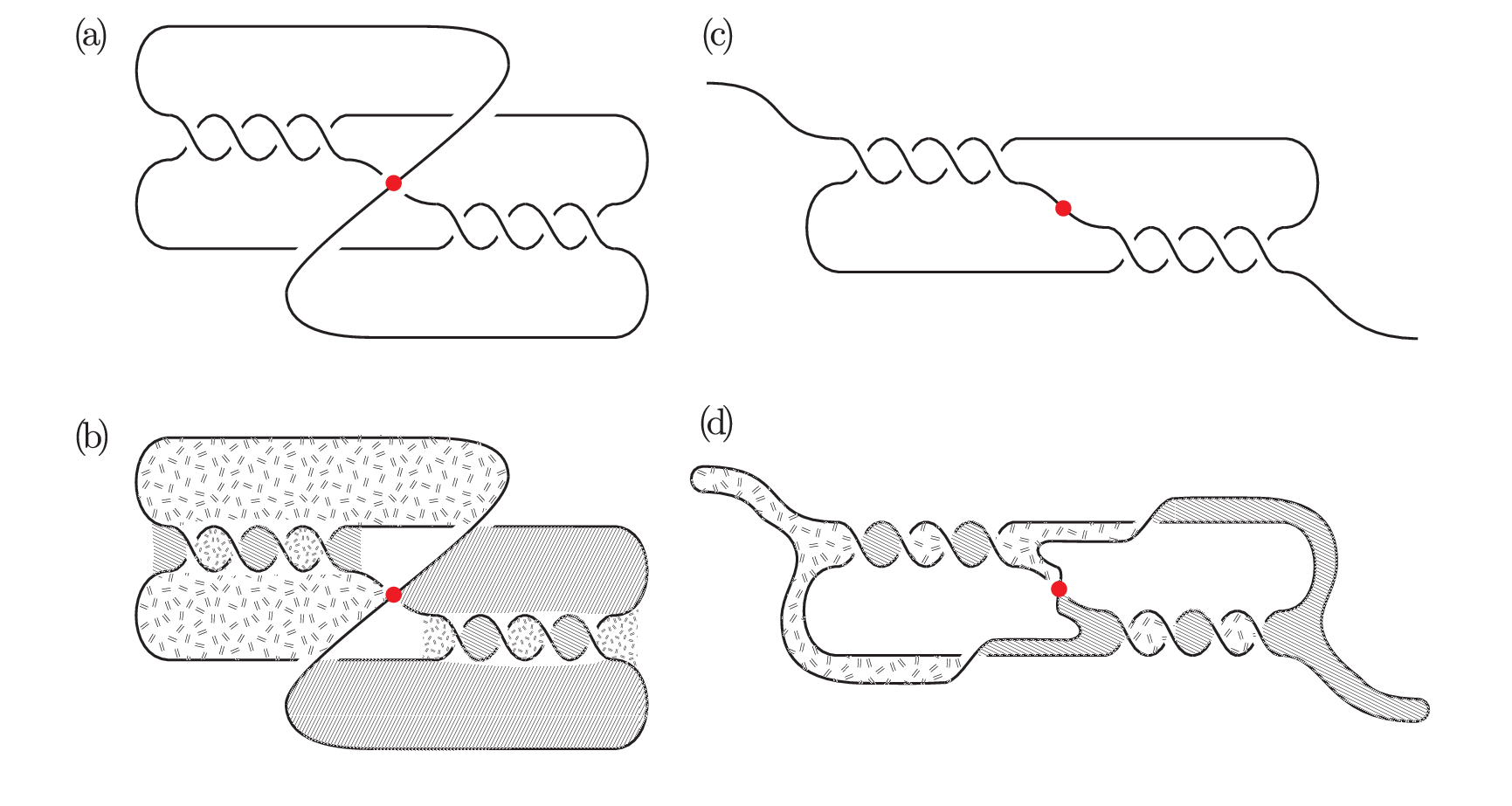}{0.95}
{fig:F1}{A transvergent diagram and a symmetric knotoid }

\subsection{Construction from an intravergent diagram and knotoid}
As Boyle and Issa note
in \cite[Proposition 1]{Boyle-Issa2021a},
Seifert's algorithm applied to an intravergent diagram 
produces an invariant Seifert surface.
To be precise, for a marked strongly invertible knot $(K,h,\delta)$,
let $\Gamma$ be an intravergent diagram 
of the strongly invertible knot $(K,h)$,
such that $\delta$ is the crossing arc 
at the crossing through which the axis passes.
Then, by applying Seifert's algorithm to $\Gamma$,  
we obtain an invariant Seifert surface for $(K,h,\delta)$
(Figure \ref{fig:F1}(b)).

By cutting the over- or under-path which contains a fixed point,
we have a rotationally symmetric knotoid diagram $\Gamma'$ (Figure \ref{fig:F1}(c)).
By applying Seifert's smoothing to $\Gamma'$, one obtains
Seifert circles and an arc.
Then replacing the arc by a thin disk, we have an invariant 
Seifert surface for $(K,h,\delta)$
(see  Figure \ref{fig:F1}(d)).
Note that theses surfaces in general have different genera.

\subsection{Construction from a transvergent diagram}
In order to construct an invariant Seifert surface
from a transvergent diagram of a strongly invertible knot $(K,h)$, 
we consider the quotient $\theta$-curve
$\theta(K,h)$ defined as follows.
Let $\pi:S^3\to S^3/h\cong S^3$ be the projection.
Then $O:=\pi(\Fix(h))$ is a trivial knot and $k:=\pi(K)$ is an arc
such that $O\cap k=\partial k$.
Thus the union $\theta(K,h):=O\cup k$ forms a $\theta$-curve embedded in $S^3$:
we call it the {\it quotient $\theta$-curve}\index{$\theta$-curve}\index{quotient!quotient $\theta$-curve} %%
of 
the strongly invertible knot $(K,h)$.
For a marked strongly invertible knot $(K,h,\delta)$,
set $\check\delta:=\pi(\delta)$ and $\check K:= k\cup \check\delta$
(Figure \ref{fig:F2}).

\figer{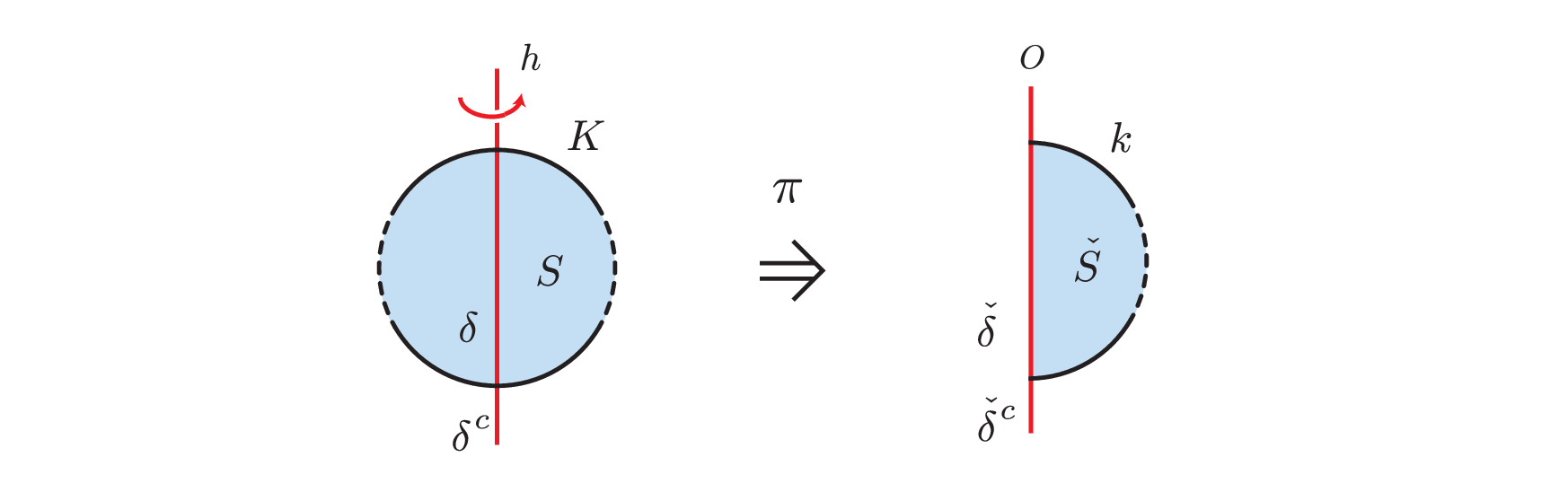}{0.95}{fig:F2}{
The quotient $\theta$-curve $\theta(K,h)=O\cup k$
and the constituent knot
$\check K=k\cup\check\delta=\pi(K\cup \delta)$.
Note that $\theta(K,h)$ consists of three edges
$k=\pi(K)$, $\check\delta=\pi(\delta)$
and $\check\delta^c=\pi(\cl(\Fix(h)\setminus \delta))=\cl(O\setminus \check\delta)$.}

Observe that if $S$ is an $h$-invariant Seifert surface for $(K,h,\delta)$,
then its image $\check S:=\pi(S)$ in $S^3/h$ is 
a {\it spanning surface}\index{surface!spanning surface}
for the knot $\check K$,
i.e., a compact surface in $S^3/h$ with boundary $\check K$,
which is disjoint from the interior of the arc 
$\check\delta^c:=\cl(O\setminus \check\delta)$.
Conversely, if $\check S$ is a spanning surface for $\check K$
which is disjoint from the interior of the arc $\check\delta^c$,
then its inverse image $\pi^{-1}(\check S)$ is 
an $h$-invariant spanning surface for $(K,h,\delta)$,
namely an $h$-invariant spanning surface for $K$
whose intersection with $\Fix(h)$ is equal to $\delta$.
However, $\pi^{-1}(\check S)$ is not necessarily orientable.
The following proposition gives a necessary and sufficient condition 
for $\pi^{-1}(\check S)$ to be orientable and so an invariant Seifert surface for $(K,h,\delta)$.
 
\begin{proposition}
\label{prop:lifting}
Under the above notation, the following hold
for every marked strongly invertible knot $(K,h,\delta)$.

{\rm (1)}
If $S$ is an invariant Seifert surface for $(K,h,\delta)$,
then its image $\check S=\pi(S)$ in $S^3/h$ 
is a spanning surface for $\check K=k\cup \check\delta$
disjoint from the interior of $\check\delta^c$,
and satisfies the following condition.
\begin{itemize}
\item[\rm{(C)}]
For any loop $\gamma$ in $\interior\check S$,
$\gamma$ is orientation-preserving or -reversing in $\check S$
according to whether the linking number
$\lk(\gamma, O)$ is even or odd. 
\end{itemize}

{\rm (2)}
Conversely, if $\check S$ is a spanning surface for the knot $\check K$ in $S^3/h$
which is disjoint from the interior of $\check\delta^c$ 
and satisfies Condition {\rm (C)},
then its inverse image $\pi^{-1}(\check S)$ in $S^3$
is an invariant Seifert surface for  $(K,h,\delta)$
\end{proposition}

\begin{proof}
(1)
Note that 
$\check S\cong S/h$ is regarded as an orbifold which has $\check\delta$ as 
a reflector line.
We consider an \lq\lq orbifold handle decomposition'' $\check S=D\cup(\cup_i B_i)$,
where 
(a) $D$ is an \lq\lq orbifold $0$-handle'', namely $D$ is a disk such that 
$\partial \check S \cap D=\check\delta$,
and (b) $\{B_i\}$ are $1$-handles attached to $D$,
namely each $B_i$ is a disk such that
$B_i\cap D =\partial B_i\cap \partial D$ 
consists of two mutually disjoint arcs 
which are disjoint from the reflector line $\check \delta$.
This handle decomposition of the orbifold $\check S$
lifts to the handle decomposition 
$S=\pi^{-1}(D)\cup (\cup_i \pi^{-1}(B_i))$ of the surface $S$,
where $\pi^{-1}(D)$ is an $h$-invariant $0$-handle 
and each $\pi^{-1}(B_i)$ consists of a pair of $1$-handles attached to the disk $\pi^{-1}(D)$.
For each $i$, $D\cup B_i$ is an annulus or a M\"obius band
which is obtained as the quotient of the orientable surface 
$\pi^{-1}(D\cup B_i)\subset S$ by the (restriction of) the involution $h$.
Let $\gamma_i$ be a core loop of $D\cup B_i$.
Then we can see that the projection $\pi:\pi^{-1}(D\cup B_i)\to D\cup B_i$ is as shown 
in Figure \ref{fig:F3}(a), (b)
according to whether $\lk(\gamma_i,O)$ is even or odd.
Hence $D\cup B_i$ is an annulus or a M\"obius band
accordingly.
Since the homology classes of $\{\gamma_i\}$ generate $H_1(\check S)$, 
this observation implies that $\check S$ satisfies Condition (C).

(2) This can be proved by reversing the above argument.
\end{proof}

\figer{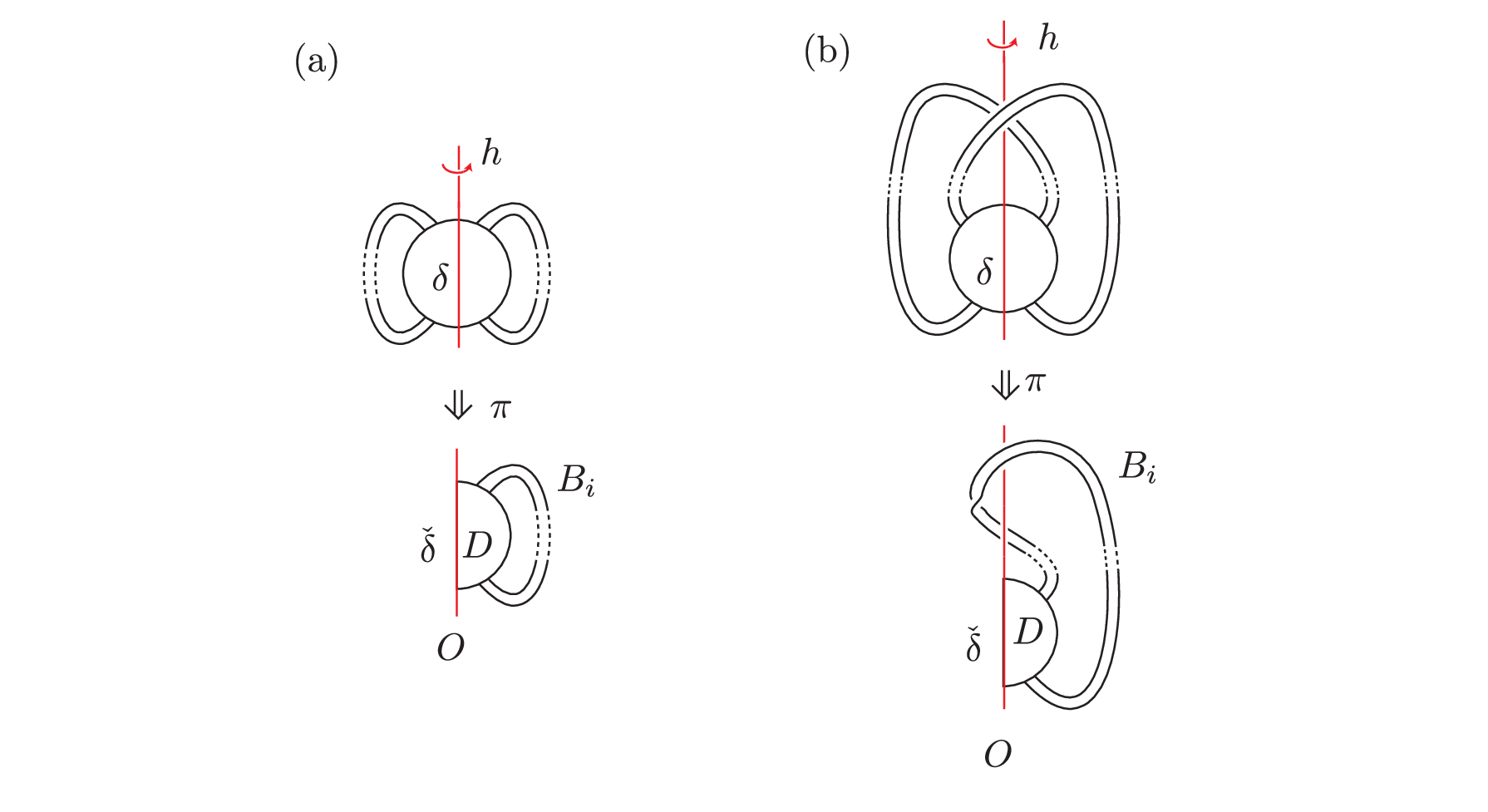}{0.8}
{fig:F3}{The band $D\cup B_i$ (bottom) and its inverse image $\pi^{-1}(D\cup B_i)$ (top).
The linking number $\lk(\gamma_i,O)$ is even in (a) and odd in (b).}

\begin{remark}
\label{rem:lifting}
{\rm
(1) Condition (C) is equivalent to the condition that
the homomorphism $\iota:H_1(\check S;\ZZ/2\ZZ) \to \ZZ/2\ZZ$
defined by $\iota(\gamma):=\lk(O,\gamma) \pmod 2$ is 
identical with the orientation homomorphism.

(2) We have $\beta_1(S)=2\beta_1(\check S)$, where $\beta_1$ denotes 
the first Betti number,
because if $b$ is the number of $1$-handles $\{B_i\}$
in the proof, then 
$\beta_1(\check S)=b$ and  $\beta_1(S)=2b$.
}
\end{remark}

From a given transvergent diagram of a strongly invertible knot $(K,h)$,
we can easily draw a diagram of the quotient $\theta$-curve $\theta(K,h)$
(see Figure \ref{fig:F5}(b)).
There are various ways of modifying the diagram into a \lq\lq good'' diagram 
from which we can construct a surface $\check S$ for $\theta(K,h)$
that satisfies the conditions in Proposition \ref{prop:lifting}. 

Hiura \cite{Hiura} gave an algorithm for obtaining such a diagram of $\theta(K,h)=O \cup k$
as follows (see Figure \ref{fig:F4}).\\
(a)
Let $k$ hook 
$O \setminus \check\delta$ in even times in the uniform way.
\\
(b) Travel along $k$ and enumerate the hooks from $1$ to $2n$. Then
slide the hooks so that they are paired and arranged from the top to the bottom along $O$.
Note that one travels between $(2i-1)^{\rm st}$ and $(2i)^{\rm th}$ hooks in 
one of the four routine ways according to the orientation of the hooks.\\
(c) Surger $k$ along the bands $\{B_i\}_{i=1}^n$ arising in the pairs of hooks
to obtain a set of arcs $k'$. 
Arrangement in (b) allows us to reset the orientation of $k'$ as depicted.

Apply  Seifert algorithm to $k' \cup \check{\delta}$ to obtain a Seifert surface $S'$.
Then attaching the bands $\{B_i\}_{i=1}^n$ to $S'$
yields a spanning surface  $\check S$ for 
$k\cup \check{\delta}$ which satisfies the 
conditions in Proposition \ref{prop:lifting}(2).
Therefore, the inverse image $S:=\pi^{-1}(\check S)$ is an invariant Seifert surface for
$(K,h,\delta)$.

\figer{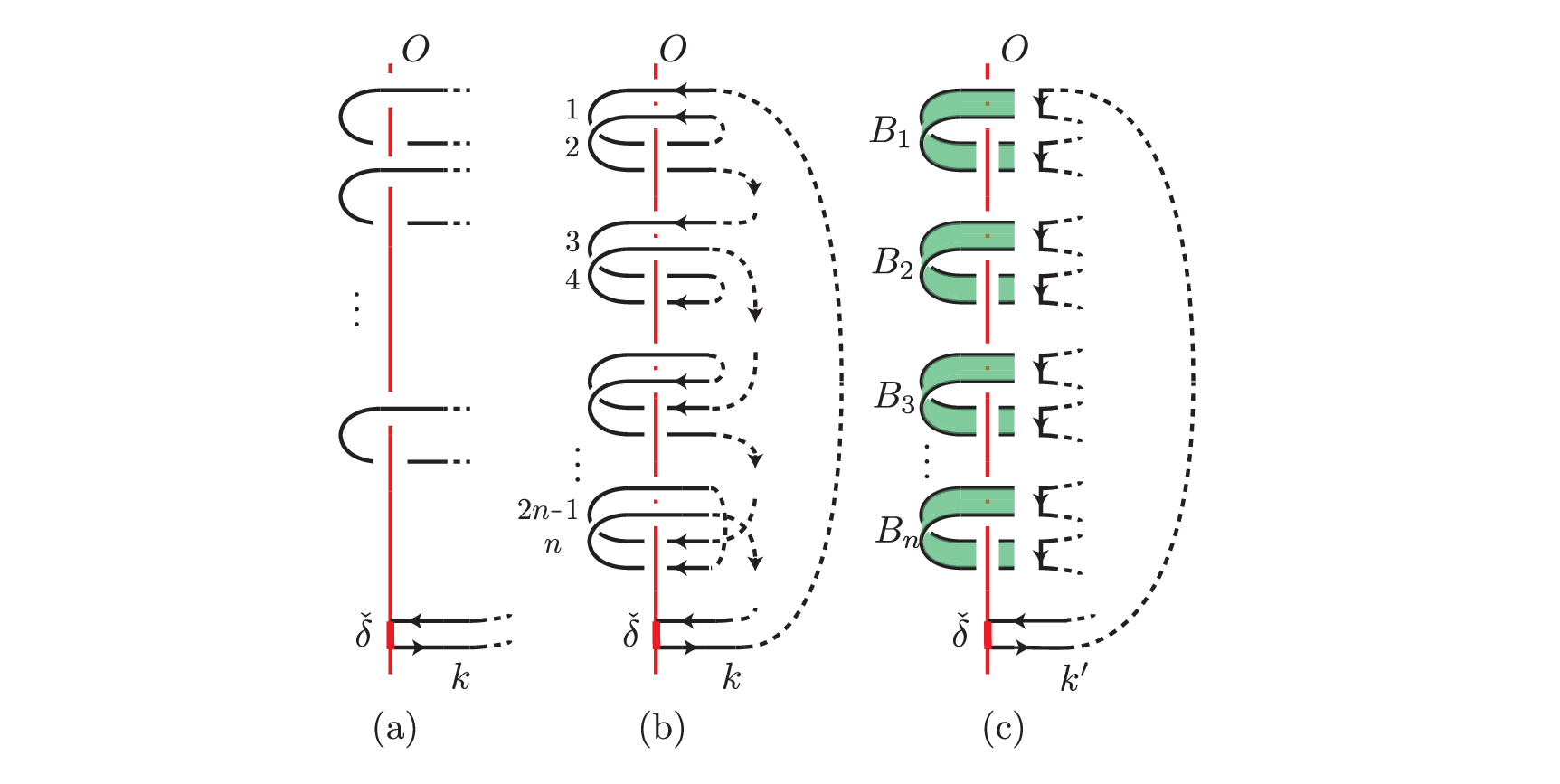}{0.9}
{fig:F4}{Hiura's algorithm to obtain an invariant Seifert surface from an intravergent diagram}

\begin{example}
\label{example:8-3}
{\rm
Let $(K,h,\delta)$ be the marked strongly invertible knot with $K=8_{3}$
as illustrated by 
the transvergent diagram in Figure \ref{fig:F5}(a).
Then by applying the algorithm to the diagram,
we obtain a genus $2$ invariant Seifert surface $S$
for $(K,h,\delta)$ as shown in Figure \ref{fig:F5}(c).
This is a minimal genus Seifert surface for $(K,h,\delta)$ 
and so $g(K,h,\delta)=2$, as shown below.
By Hatcher-Thurston \cite[Theorem 1] {Hatcher-Thurston}
or by Kakimizu \cite{Kakimizu2005} succeeding the work of Kobayashi \cite{Kobayashi},
$8_{3}$ has precisely two genus $1$ Seifert surfaces up to equivalence.
(They are obtained by applying Seifert's algorithm to the alternating diagram,
where there are two different ways of attaching a disk to the unique big Seifert circle.)
Obviously they are interchanged by $h$, and hence not $h$-invariant.

Kakimizu \cite{Kakimizu2005} in fact showed that the two genus $1$ Seifert surfaces
are the only incompressible Seifert surfaces for $8_{3}$.
Thus $8_{3}$ does not even admit an $h$-invariant incompressible Seifert surface.
So the result of Edmonds and Livingston \cite[Corollary 2.2]{Edmonds-Livingston}
that every periodic knot admits an invariant incompressible Seifert surface 
does not hold for strongly invertible knots.
}
\end{example}

\figer{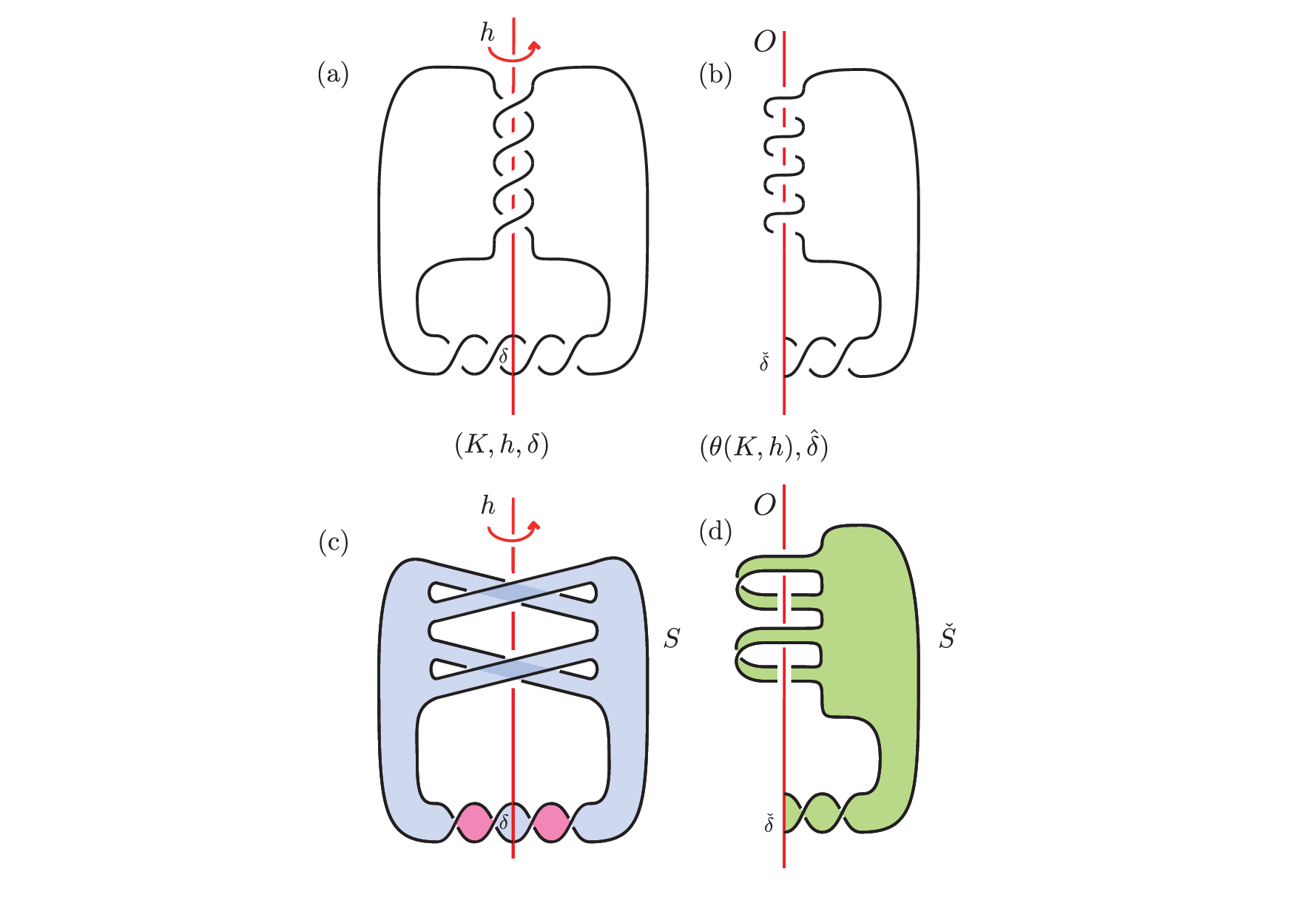}{0.99}
{fig:F5}{A transvergent diagram for $8_{3}$ and an invariant Seifert surface}

\section{Proof of Theorem \ref{thm:main}}
\label{sec:Proof-MainTheorem}

By Proposition \ref{prop:lifting}
and Remark \ref{rem:lifting},
we can characterize the equivariant genus $g(K,h,\delta)$ 
in terms of the quotient $\theta$-curve $\theta(K,h)=k\cup \check \delta\cup \check\delta^c$
and the constituent knot $\check K=k\cup \check \delta$ as follows.

\begin{proposition}
\label{prop:band-number}
Let $(K,h,\delta)$ be a marked strongly invertible knot.
Then $g(K,h,\delta)$ is equal to the minimum of
$\beta_1(\check S)$ where 
$\check S$ runs over the spanning surfaces for 
the constituent knot $\check K=k\cup \check \delta$ of
the quotient $\theta$-curve $\theta(K,h)$
that is
disjoint from the interior of the remaining edge  $\check\delta^c$ 
and satisfies Condition (C).
\end{proposition}

By relaxing the definition of the crosscap number\index{crosscap number}
, $\gamma(K)$, of a knot $K$
introduced by Clark \cite{Clark},
we define the {\it band number}\index{band number}, 
$b(K)$, to be the 
minimum of $\beta_1(G)$
of all spanning surfaces $G$ for $K$
(see Murakami-Yasuhara \cite{Murakami-Yasuhara}).
In other words,  $b(K)=\min(2g(K),\gamma(K))$.
Then we have the following corollary.

\begin{corollary}
\label{cor:band-number}
$g(K,h,\delta)\ge b(\check K)$.
\end{corollary}

For any marked strongly invertible knot $(K,h,\delta)$
with $K$ a $2$-bridge knot,
the constituent knot $\check K$ is either the trivial knot or a $2$-bridge knot
(see \cite[Proposition 3.6]{Sakuma1986}).
In his master thesis \cite{Bessho} supervised by the last author,
Bessho described a method for determining the cross cap numbers
of $2$-bridge knots by using the result of Hatcher and Thurston
\cite[Theorem 1(b)] {Hatcher-Thurston}
that classifies the
incompressible and boundary incompressible surfaces in the $2$-bridge knot exteriors.
Hirasawa and Teragaito \cite{Hirasawa-Teragaito}
promoted the method into a very effective algorithm.
For some classes of marked strongly invertible $2$-bridge knots,
we can apply Corollary \ref{cor:band-number} by using
Hirasawa-Teragaito method, as shown in the following example.

\begin{example}
\label{example:main-example}
{\rm
For a positive integer $n$,
let $K_n$ be the the plat closure of a 4-string braid
$(\sigma_2^2 \sigma_1^4)^n$.
Then $K_n$ is the $2$-bridge knot\index{$2$-bridge knot}\index{knot!$2$-bridge knot} 
whose slope $q/p$
has the continued fraction expansion $[2, 4, 2,4, \cdots,2,4]$ of length $2n$. 
Here we employ the following convention of continued fraction expansion, 
which is used in \cite{Hirasawa-Teragaito}.
\[
[a_1,a_2, \cdots ,a_m]:=\cfrac{1}
{a_1-\cfrac{1}
{a_2-\cfrac{1}
{\ddots -\cfrac{1}
{a_{m}
}}}}
\]
Thus $K_n$ is the boundary of a linear plumbing of unknotted annuli 
where the $i^{\rm th}$ band has $2$ or $4$ right-handed half-twists 
according to whether $i$ is odd or even
(see \cite[Fig. 2]{Hatcher-Thurston}).
In particular, $g(K_n)=n$.
Note that $K_n$ is isotopic to the knot 
admitting the strong inversion $h$ in Figure \ref{fig:F6}(a).
Let $(K_n,h,\delta)$ be the marked strongly invertible knot,
where $\delta$ is the long arc in $\Fix(h)$ bounded by $\Fix(h)\cap K_n$
illustrated in Figure \ref{fig:F6}(a).
Then we have
$g(K_n,h,\delta)=2n=g(K_n)+n$, as explained below.

\figer{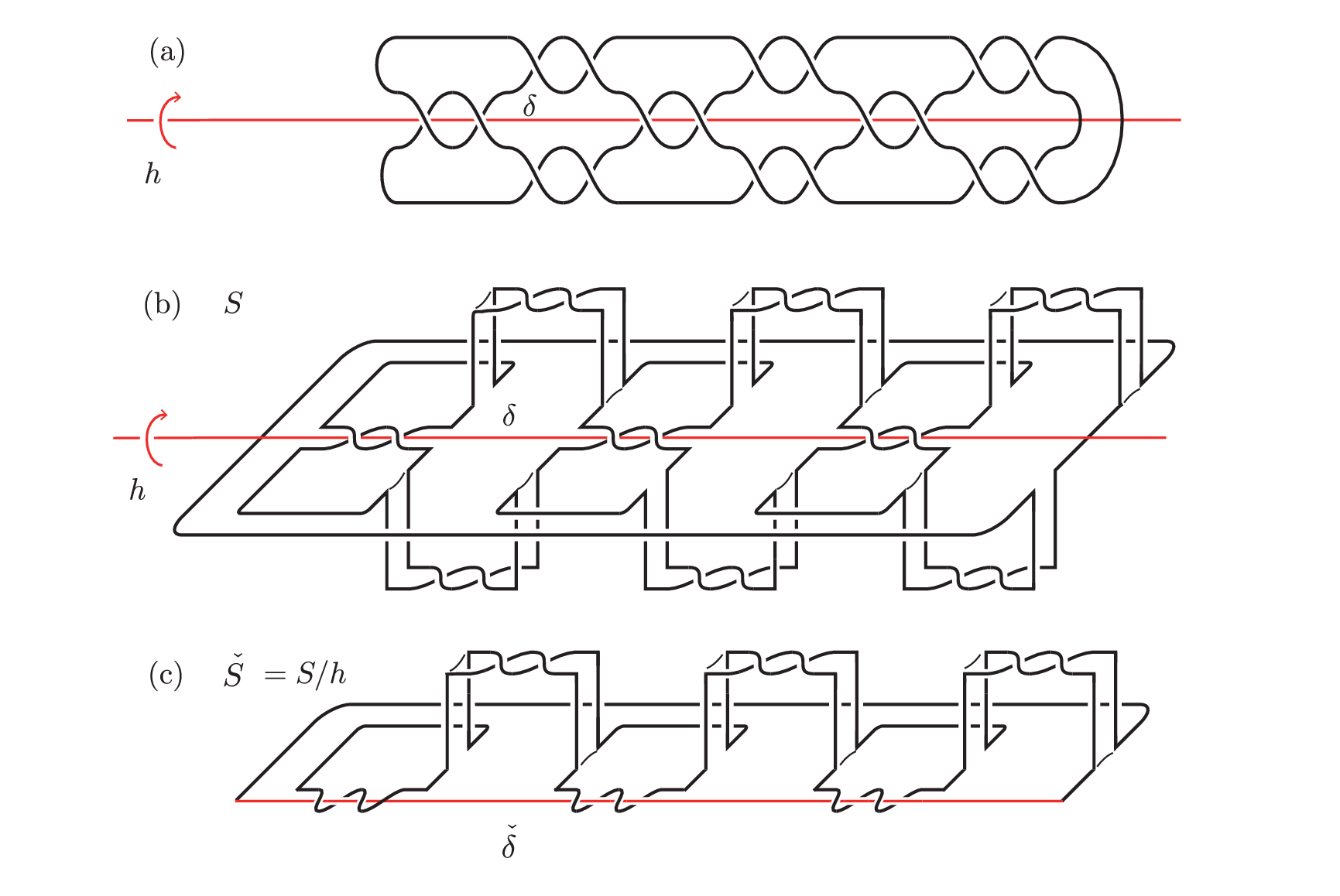}{1}{fig:F6}{A Seifert surface realizing
the equivariant genus}

Observe that $(K_n, h, \delta)$ is equivalent to the marked strongly invertible knot
bounding the $h$-invariant Seifert surface $S$ of genus 
$2n$ in Figure \ref{fig:F6}(b). 
Consider the quotient surface $\check S=S/h$, 
which is a spanning surface for the knot $\check K:=k\cup \check \delta$.
We show that the band number $b(\check K)$ is equal to $\beta_1(\check S)$.
Then it follows from Corollary \ref{cor:band-number} 
that $S$ is a minimal genus Seifert surface for $(K_n, h, \delta)$. 
To this end, observe that $\check K$ is the $2$-bridge knot
whose slope has the continued fraction expansion 
$\mathcal{C}:=[4,2,4,2,\cdots,4,2]$ of length $2n$ 
(Figure  \ref{fig:F6}(c)).
Since $\mathcal{C}$ does not contain an entry ${0, -1,1}$
and since $\mathcal{C}$ does not contain a sub-sequence $\pm [\cdots, 2, 3, \cdots, 3, 2,\cdots]$
nor $\pm [\cdots, 2, 2,\cdots]$,
it follows from
\cite[Theorems 2 and 3]{Hirasawa-Teragaito}, 
that the length $2n$ of $\mathcal{C}$ is
minimal among all continued fraction expansions of
all fractions representing $\check K$.
Then it follows from \cite[Theorem 1(b)] {Hatcher-Thurston} 
that $b(\check K)=\beta_1(\check S)=2n$ as desired.
}
\end{example}

Theorem \ref{thm:main} follows from the above example
and Proposition \ref{prop:fibered}, which is proved at the end of 
Section \ref{sec:proof-SubMainTheorem}.
Though the above method is also applicable to a certain family of
marked strongly invertible $2$-bridge knots,
there are various cases
where the above simple method does not work.
In fact, there is a case where the knot $\check K=k\cup\check\delta$
is trivial (see Figure \ref{fig:F5}).
The second author \cite{Hiura} treated such a family by applying 
the method in \cite{Hatcher-Thurston} 
to such spanning surfaces as in Proposition \ref{prop:band-number}.
However, there still remained cases where none of the methods described above work.
In the sequel of this paper [28], 
we give a unified determination of the equivariant genera of 
all marked strongly invertible 2-bridge knots using sutured manifolds.
Theorem \ref{thm:disjoint-Seifert2} enabled us to do that without 
invoking \cite{Hatcher-Thurston} 
or \cite{Hirasawa-Teragaito}.

\section{Proof of Theorems \ref{thm:disjoint-Seifert} and \ref{thm:disjoint-Seifert2}}
\label{sec:proof-SubMainTheorem}

For a knot $K$ in $S^3$, let 
$E(K):=S^3\setminus \interior N(K)$ be the {\it exterior}\index{exterior}\index{knot!knot exterior} 
of $K$,
where  $N(K)$ is a regular neighborhood of $K$.
If $F$ is a Seifert surface for $K$,
then after an isotopy, 
$F$ intersects $N(K)$ in a collar neighborhood of $\partial F$,
and $F\cap E(K)$ is a surface properly embedded in $E(K)$
whose boundary is a {\it preferred longitude}\index{preferred longitude}, 
i.e., a longitude of the solid torus $N(K)$
whose linking number with $K$ is $0$.
Conversely, any such surface in $E(K)$ determines a 
Seifert surface\index{Seifert surface}\index{surface!Seifert surface} 
for $K$.
Thus we also call such a surface in $E(K)$ a {\it Seifert surface} for $K$.

\begin{proof}[Proof of Theorem \ref{thm:disjoint-Seifert}]
We give a proof imitating the arguments by Edmonds \cite{Edmonds}.
Let $(K,h)$ be a strongly invertible knot, and 
let $E:=E(K)$ be an $h$-invariant exterior of $K$.
We continue to denote by $h$ the restriction of $h$ to $E$. 
Then Theorem \ref{thm:disjoint-Seifert} is equivalent to the existence
of a minimal genus Seifert surface $F\subset E$ for $K$ such that $F\cap h(F)=\emptyset$.

Fix an $h$-invariant Riemannian metric on $E$  
such that $\partial E$ is convex.
Choose a preferred longitude $\ell_0\subset \partial E$ of $K$
such that $\ell_0$ and $\ell_1:=h(\ell_0)$ are disjoint.
Let $F^*$ be a smooth, compact, connected, orientable surface of genus $g(K)$ with 
one boundary component,
$\psi_0:\partial F^* \to \partial E$ an embedding such that $\psi_0(\partial F^*)=\ell_0$,
and set $\psi_1:=h \psi_0$.
For $i=0,1$, let $\FF_i$ be the space of all piecewise smooth maps
$f:(F^*,\partial F^*)\to (E,\ell_i)$
properly homotopic to an embedding, such that $f|_{\partial F^*}=\psi_i$.
Then we have the following \cite[Proposition 1]{Edmonds}.

\begin{lemma}
Each
$\FF_i$ contains an area minimizer\index{area minimizer},
namely, 
there exists an element $f_i\in \FF_i$ 
whose area 
is minimum among the areas of all elements of $\FF_i$.
Moreover, any area minimizer in $\FF_i$ is an embedding.
\end{lemma} 

Since $h$ is an isometric involution, it follows that
$f_i$ is an area minimizer in $\FF_i$ if and only if 
$h f_i$ is an area minimizer in $\FF_j$,
where $\{i,j\}=\{0,1\}$.
Thus Theorem \ref{thm:disjoint-Seifert} follows from the 
following analogue of \cite[Theorem 2]{Edmonds}.
\end{proof}

\begin{theorem}
\label{thm:Edomonds-Thm2}
For $i=0,1$,
let $f_i$ be an area minimizer in $\FF_i$,
and $F_i\subset E$ the minimal genus Seifert surface for $K$ obtained as the image of $f_i$.
Then $F_0$ and $F_1$ are disjoint.
\end{theorem}

\begin{proof}
The proof is the same as \cite[Proof of Theorem 2]{Edmonds}, as explained below.
Suppose to the contrary that $F_0\cap F_1\ne \emptyset$.
We first assume that $F_0$ and $F_1$ intersect transversely.
Then, by 
the arguments in \cite[the 3rd to the 6th paragraphs of the proof of Theorem 2]{Edmonds}
using the area minimality and the incompressibility of $F_i$ ($i=0,1$) and
the asphericity of $E$,
it follows that every component of $F_0\cap F_1$ is 
essential in both $F_0$ and $F_1$.

By the arguments in 
\cite[the 7th to the final paragraphs of the proof of Theorem 2]{Edmonds},
we see that there is a submanifold $W$ in $E$ 
satisfying the following conditions.
\begin{enumerate}
\item[(a)]
$\partial W=A\cup B$ where $A=W\cap F_0$ and  $B=W\cap F_1$.
\item[(b)]
Both $(F_0\setminus A)\cup B$ and $(F_1\setminus B)\cup A$
are minimal genus Seifert surfaces.
\end{enumerate}
Since $(F_0\setminus A)\cup B$ and $(F_1\setminus B)\cup A$ have corners,
smoothing them reduces area and yields two minimal genus Seifert surfaces,
at least one of which has less area than $F_0$ or $F_1$,
a contradiction.

Finally, we explain the generic case where 
$F_0$ and $F_1$ are not necessarily transversal.
As is noted in \cite[the 2nd paragraph of the proof of Theorem 2]{Edmonds},
by virtue of the Meeks-Yau trick\index{Meeks-Yau trick}
introduced in \cite{Meeks-Yau} and 
discussed in \cite{Freedman-Hass-Scott},
we can reduce to the case where $F_0$ and $F_1$ intersect transversely as follows.
By \cite[Lemma 1.4]{Freedman-Hass-Scott}, we have a precise picture of the situation
where $F_0$ and $F_1$ are non-transversal
(see \cite[Figure 1.2]{Freedman-Hass-Scott}).
By using this fact, $F_0$ is perturbed slightly to a surface $F_0'$, so that
(a) $F_0'$ and $F_1$ intersect transversely and nontrivially, and 
(b) the difference $\Area(F_0')-\Area(F_0)$ is less than a lower bound estimate for the area reduction to be achieved as in the preceding paragraphs. 
Thus the assumption $F_0\cap F_1\ne \emptyset$ again leads to a contradiction.
\end{proof}

\begin{proof}[Proof of Theorem \ref{thm:disjoint-Seifert2}]
Throughout the proof we use the following terminology:
for a topological space $X$ and its subspace $Y$, 
a {\it closed-up component}\index{closed-up component} of $X\setminus Y$
means the closure in $X$ of a component of $X\setminus Y$.
Let $(K,h,\delta)$ be a marked strongly invertible knot
and let $E$ and $h$ be as 
in the proof of Theorem \ref{thm:disjoint-Seifert}.
Then $\Fix(h)\subset E$ is the disjoint union of two arcs $\delta\sqcup\delta^c$,
where $\delta$ denotes the intersection of the original $\delta$ with $E$
and $\delta^c=\Fix(h)\setminus \delta$.
Then, by the assumption of Theorem \ref{thm:disjoint-Seifert2},
there is 
a minimal genus Seifert surface $F\subset E$  
such that $F\cap h(F)=\emptyset$.
Let $E_{\delta}$ and $E_{\delta^c}$ be the closed-up components of 
$E(K)\setminus (F\cup h(F))$ containing $\delta$ and 
$\delta^c$, respectively.
Then there is a minimal genus Seifert surface $S \subset E$ for $(K,h,\delta)$
such that $\partial S$ is 
contained in the interior of the annulus $\partial E\cap \partial E_{\delta}$.
We will construct from $S$ 
a minimal genus Seifert surface for $(K,h,\delta)$
which is properly embedded in 
$E_{\delta}\setminus (F\cup h(F))$.

\begin{claim}
\label{claim:essential-intersection0}
We can choose $S$ so that
$S$ intersects $F\cup h(F)$ transversely and so that
every component of $S\cap(F\cup h(F))$ is essential
in both $S$ and $F\cup h(F)$.
\end{claim}

\begin{proof}
Since $F\cap h(F)=\emptyset$,
we can $h$-equivariantly isotope $S$ so that 
it intersects $F\cup h(F)$ transversely.
Then $S\cap(F\cup h(F))$ consists of simple loops,
because the boundaries of $S$ and $F\cup h(F)$ are disjoint.

Suppose first that $S\cap(F\cup h(F))$ contains a component
that is inessential in $F\cup h(F)$.
Let $\alpha$ be one such component which is innermost in $F\cup h(F)$,
and let $\Delta$ be the disk in $F\cup h(F)$ bounded by $\alpha$.
Then $\Delta\cap S=\partial\Delta$, $h(\Delta)\cap S=\partial (h(\Delta))$, and
$\Delta\cap h(\Delta)=\emptyset$.
Let $S'$ be the $h$-invariant surface obtained from $S$
by surgery along $\Delta\cup h(\Delta)$,
i.e., $S'$ is obtained from $S$ by removing 
an $h$-invariant open regular neighborhood of 
$\partial(\Delta\cup h(\Delta))$
in $S$
and capping off the resulting four boundary circles
with nearby parallel $h$-equivariant copies of $\Delta$ and $h(\Delta)$.
Let $S'_b$ be the component of $S'$ containing $\partial S$.
Then $S'_b$ is a Seifert surface for $(K,h,\delta)$
such that $g(S'_b)\le g(S)=g(K,h,\delta)$. 
Hence $S'_b$ is also a minimal genus Seifert surface for $(K,h,\delta)$. 
Moreover $|S'_b\cap(F\cup h(F))|\le |S\cap(F\cup h(F))|-2$,
where $|\cdot|$ denotes the number of the connected components.

Suppose next that $S\cap(F\cup h(F))$ contains a component
that is inessential in $S$.
Let $\alpha$ be one such component which is innermost in $S$.
Then $\alpha$ is also inessential in $F\cup h(F)$
by the incompressibility of  $F\cup h(F)$.
So, by the argument in the preceding paragraph,
we obtain a minimal genus Seifert surface $S'_b$ 
for $(K,h,\delta)$ such that
$|S'_b\cap(F\cup h(F))|\le |S\cap(F\cup h(F))|-2$.

By repeating the above arguments,
we can find a desired  minimal genus Seifert surface for $(K,h,\delta)$.
\end{proof}

Let $p:\tilde E\to E$ be the infinite cyclic covering,
and let $E_j$ ($j\in\ZZ$) be the closed-up components of 
$\tilde E\setminus p^{-1}(F\cup h(F))$
satisfying the following conditions.
\begin{enumerate}
\item
$\tilde E=\cup_{j\in\ZZ} E_j$.
\item
$E_j$ projects homeomorphically onto $E_{\delta}$ or $E_{\delta^c}$
according to whether $j$ is even or odd.
\item
$F_j:=E_{j-1}\cap E_j$ 
projects homeomorphically onto $F$ or $h(F)$
according to whether $j$ is even or odd. 
\end{enumerate}
For each $j\in\ZZ$,
let $h_j$ be the involution on $\tilde E$ 
which is a lift of $h$ such that $h_j(E_j)=E_j$.
Note that $h_j(E_i)=E_{2j-i}$ and that $\Fix(h_j)$ is a properly embedded arc in $E_j$
which projects to $\delta$ or $\delta^c$ according to whether $j$ is even or odd.
The composition 
$\tau:=h_{j+1}h_j$ is independent of $j$, and
gives a generator of the covering transformation group of $\tilde E$,
such that $\tau(E_i)=E_{i+2}$ and $\tau(F_i)=F_{i+2}$ for every $i$.

Note that $p^{-1}(S)=\sqcup_{i\in\ZZ}S_{2i}$,
where $S_{2i}$ is the lift of $S$ preserved by $h_{2i}$.
Let $r:=\max\{j\in\ZZ \ | \ S_0\cap E_j\ne \emptyset\}$, and 
assume that the number $r$
is minimized among all minimal genus Seifert surfaces $S$ for $(K,h,\delta)$
that satisfy the conclusion of Claim \ref{claim:essential-intersection0}.
Then Theorem \ref{thm:disjoint-Seifert} is equivalent to the assertion that $r=0$.

Suppose to the contrary that $r>0$. 
Let $S_0^-$ be the closed-up component of $\tilde E\setminus S_0$
containing $S_{-2}=\tau^{-1}(S_0)$.
Set $\tilde W:=S_0^- \cap E_r$ and $W:=p(\tilde W)$.
Since $\tilde W\subset E_r$,
$p|_{\tilde W}:\tilde W\to W$ is a homeomorphism.
(See Figure \ref{fig:F7}.)

\figer{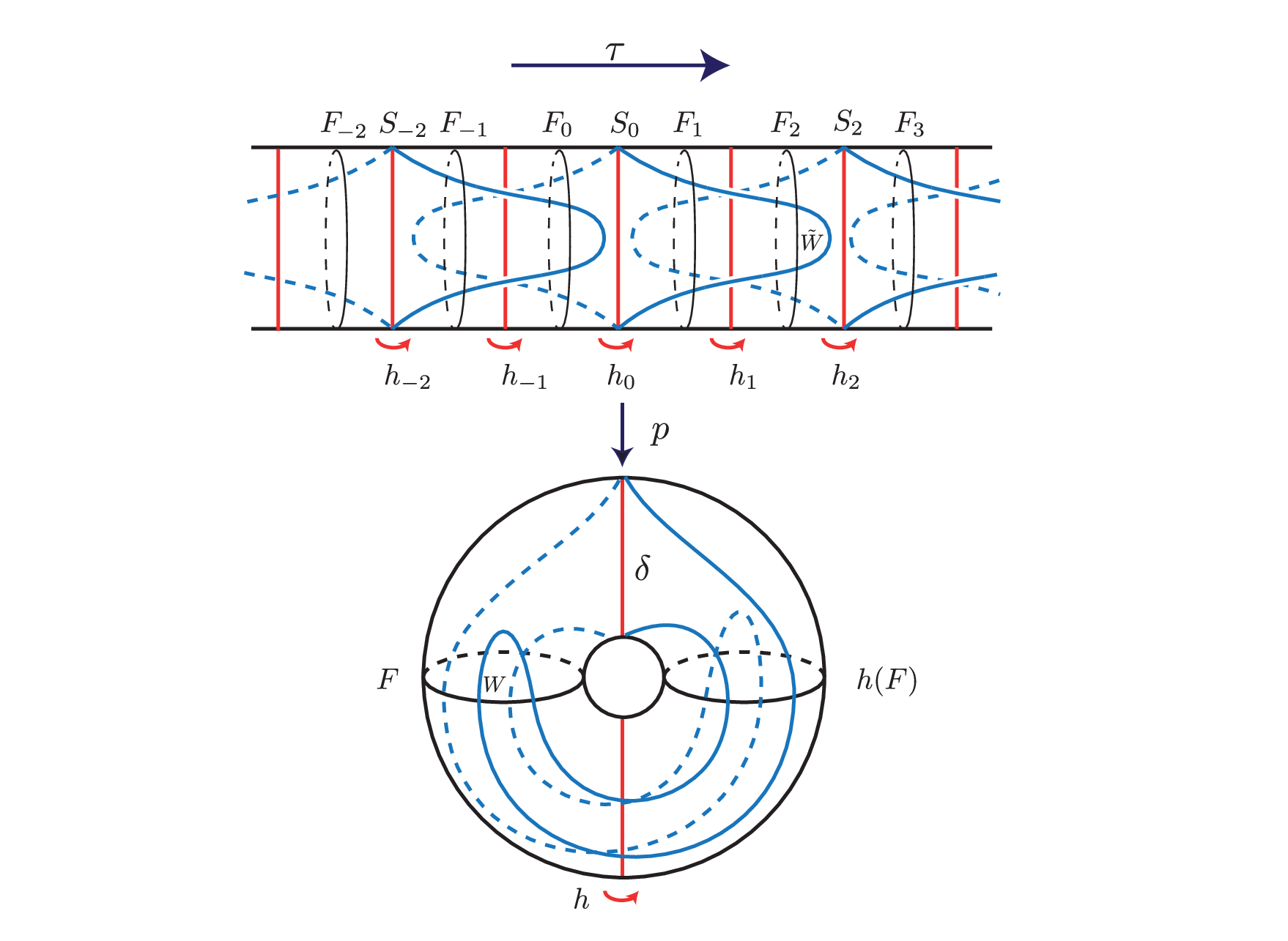}{0.95}{fig:F7}{Schematic picture of $p:\tilde E\to E$, where $r=2$.
The picture does not reflect the assumption that
the boundaries of $F$, $h(F)$ and $S$ are disjoint.}

\begin{claim}
\label{claim:W}
$W$ is a (possibly disconnected) compact $3$-manifold contained in $\interior E$
such that 
$\partial W=A \cup B$,
where 
$A=S\cap W$, 
$B=p(F_r) \cap W$,
and $A \cap B= p(S_0\cap F_r)\subset S\cap p(F_r)$.
Here $p(F_r)=F$ or $h(F)$ according to whether $r$ is even or odd.
\end{claim}

\begin{proof}
Note that $S_0^-\cap \partial \tilde E$ is the half-infinite annulus in $\partial \tilde E$
which forms the closed-up component of $\partial \tilde E \setminus \partial S_0$
disjoint from $\partial F_1$.
On the other hand, $E_r\cap \partial \tilde E$ is the annulus in $\partial \tilde E$
bounded by $\partial F_r$ and $\partial F_{r+1}$.
Since $r>0$ by the assumption, these imply that 
$S_0^-\cap \partial \tilde E$ and $E_r\cap \partial \tilde E$ are disjoint.
Hence $\tilde W$ is disjoint from $\partial \tilde E$ and
so $\tilde W \subset \interior\tilde E$.
Note that $\fr S_0^-=S_0$ and $\fr E_r = F_r \sqcup F_{r+1}$
intersect transversely (Claim \ref{claim:essential-intersection0}),  where $S_0^- \cap F_{r+1}=\emptyset$. 
(Here $\fr Y$ with $Y=S_0^-, E_r$ denotes the frontier\index{frontier} 
of $Y$ in $\tilde E$,
namely the closure of $Y$ in $\tilde E$ minus the interior of $Y$
in $\tilde E$.) 
Hence $\tilde W$ is a compact $3$-manifold contained in $\interior \tilde E$,
such that 
$\partial \tilde W=\tilde A \cup \tilde B$,
where $\tilde A=(\fr S_0^-) \cap \tilde W=S_0\cap\tilde W$,
$\tilde B=(\fr E_r) \cap \tilde W=F_r \cap \tilde W$, and 
$\tilde A\cap \tilde B
= (S_0\cap\tilde W) \cap  (F_r \cap \tilde W)= S_0\cap F_r$.
Since $p|_{\tilde W}:\tilde W\to W$ is a homeomorphism,
these imply the claim.
\end{proof}

\begin{claim}
\label{claim:Wh}
$W\cap h(W)=\emptyset$.
\end{claim}

\begin{proof}
Since the restriction $p|_{E_r}$ is a homeomorphism
onto its image $p(E_r)$
(which is equal to $E_{\delta}$ or $E_{\delta^c}$ according to whether $r$ is even or odd)
and since the involution $h|_{p(E_r)}$ is pulled back to
the involution $h_r|_{E_r}$ by $p|_{E_r}$,
it follows that $p|_{E_r}$ 
restricts to a homeomorphism from $\tilde W \cap h_r(\tilde W)$
onto $W\cap h(W)$.
On the other hand, we have
\[
\tilde W \cap h_r(\tilde W) 
=
(S_0^- \cap E_r) \cap h_r(S_0^- \cap E_r)
=
(S_0^-\cap h_r(S_0^-))\cap E_r 
=
\emptyset \cap E_r=\emptyset,
\]
where the third identity is verified as follows.
Recall that $S_0^-$ is the closed-up component of $\tilde E\setminus S_0$
containing $S_{-2}$.
Thus $h_r(S_0^-)$ is the closed-up component of 
$\tilde E\setminus h_r(S_0)=\tilde E\setminus S_{2r}$
containing $h_r(S_{-2})=S_{2r+2}$.
Since $r>0$, this implies $S_0^-\cap h_r(S_0^-)=\emptyset$.
Hence $W\cap h(W)= p(\tilde W \cap h_r(\tilde W)) =\emptyset$ as desired.
\end{proof}

We perform an $h$-equivariant cut and paste operation on $S\cup F\cup h(F)$ along $W\cup h(W)$,
and produce surfaces $S'=S'_b\sqcup S'_c$, $F'=F'_b\sqcup F'_c$ and 
$h(F')=h(F'_b)\sqcup h(F'_c)$ as follows.
\begin{enumerate}
\item
$S':=(S\setminus (A\cup h(A))\cup(B\cup h(B))$,
$S'_b$ is the component of $S'$
containing $\partial S$, and $S'_c:=S'\setminus S'_b$.
\item
$F':= (F\setminus B) \cup A$,
$F'_b$ is the component of $F'$ containing $\partial F$,
and $F'_c:=F'\setminus F'_b$.
\end{enumerate}
Then $S'_b$ is a Seifert surface for $(K,h,\delta)$, and 
both $F'_b$ and $h(F'_b)$ are Seifert surfaces for $K$.
Let $\Sigma$ be the disjoint union of copies of 
$S'_c$, $F'_c$ and $h(F'_c)$.
Then $\Sigma$ is a possibly empty, closed, orientable surface,
such that $\chi(\Sigma)=\chi(S'_c\cup F'_c\cup h(F'_c))$.
(Note that the intersection among $S'_c$, $F'_c$ and $h(F'_c)$
consists of disjoint loops.)
Since $S$ and $F\cup h(F)$
intersects in essential loops by Claim \ref{claim:essential-intersection0},
none of the components of $\Sigma$ is a $2$-sphere
and therefore $\chi(\Sigma)\le 0$.
Hence
\begin{align*}
\chi(S'_b)+\chi(F'_b)+\chi(h(F'_b))
& \ge  \chi(S'_b)+\chi(F'_b)+\chi(h(F'_b))+\chi(\Sigma)\\
& = \chi(S) +\chi(F)+\chi(h(F)).
\end{align*}
Since $g(F)=g(K)\le g(F'_b)$, this implies
\[
\chi(S'_b) \ge
\chi(S)+(\chi(F)-\chi(F'_b))+(\chi(h(F))-\chi(h(F'_b)))
\ge \chi(S).
\]
Hence $g(S'_b) \le g(S)=g(K,h,\delta)$, and therefore
$S'_b$ is also a minimal genus Seifert surface for $(K,h,\delta)$.

Note that the $h_0$-invariant lift of $S'_b$ is
obtained from $S_0 \subset \tilde E$ by cut and paste operation along 
$\tilde W \cup h_0(\tilde W)$ 
and so it is
contained in the region of $\tilde E$
bounded by $F_{1-r}$ and $F_r$.
After a small $h$-equivariant isotopy of $S'_b$,
the lift is contained in the interior of that region,
and hence the number $r$ for $S'_b$ is strictly smaller than the original $r$.
This contradicts the minimality of $r$.
Hence we have $r=0$ as desired.
\end{proof}

The following example illustlates
Theorems \ref{thm:disjoint-Seifert}, \ref{thm:disjoint-Seifert2}.
and Corollary \ref{cor:disjoint-Seifert}.

\figer{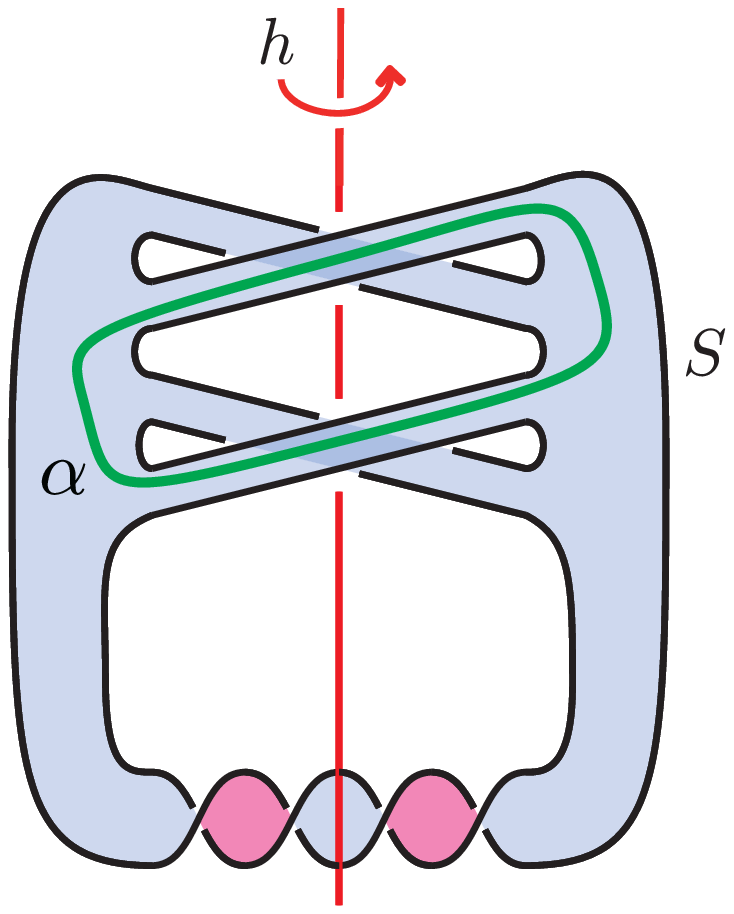}{0.9}{fig:F8}{The invariant Seifert surface $S$ and its compressing loop $\alpha$
on the positive side}

\begin{example}
\label{example:compression}
{\rm
Let $S$ be the invariant Seifert surface for $K=8_3$ 
constructed in Example \ref{example:8-3}.
Then the loop $\alpha$ in Figure \ref{fig:F8} is a compressing loop of $S$
on the positive side,
and $h(\alpha)$ is a compressing loop of $S$
on the negative side.
Note that $\alpha$ and $h(\alpha)$ intersect nontrivially,
and therefore $\{\alpha, h(\alpha)\}$ does not yield 
an $h$-equivariant compression of $S$.
Consider a parallel copy $S_+$ of $S$ on the positive side, 
and let $F_+$ be the minimal genus Seifert surface
obtained by compressing $S_+$ along a parallel copy of $\alpha$ on $S_+$.
Then $F_-:=h(F_+)$ is obtained from the parallel copy $h(S_+)$ of $S$ on the negative side
through compression along a parallel copy of $h(\alpha)\subset h(S_+)$.
The minimal genus Seifert surfaces $F_+$ and $F_-=h(F_+)$ are disjoint, and $S$ lies in 
a region between them.

As noted in Example \ref{example:8-3}, $K=8_3$ has precisely two 
minimal genus Seifert surfaces up to equivalence,
and the Kakimizu complex $MS(K)$ consists of a single edge and two vertices.
The automorphism $h_*$ of $MS(K)$ induced by $h$ is the reflection in the center
of the $1$-simplex. 
}
\end{example}

\medskip

Finally,
we prove Proposition \ref{prop:fibered}
for fibered knots.

\begin{proof}[Proof of Proposition \ref{prop:fibered}]
Let $p:E(K) \to S^1$ be the fibering whose fibers are 
minimal genus Seifert surfaces for $K$.
By the result of Tollefson \cite[Theorem 2]{Tollefson},
we may assume that the involution $h$ on $E(K)$ preserves the fibering
and that the involution $\check h$ on $S^1$ induced from $h$
is given by $\check h(z)=\bar z$,
where we identify $S^1$ with the unit circle on the complex plane. 
Then the inverse image $p^{-1}(\pm 1)$ gives a pair of
$h$-invariant Seifert surfaces for $K$ with genus $g(K)$.
Since these give Seifert surfaces
for the two marked strongly invertible knots associated with $(K,h)$,
we have $g(K,h,\delta)=g(K)$ as desired.
\end{proof}

\section{Equivariant $4$-genus}\label{sec:fourgenus}
In this section, 
we review old and new studies of
equivariant $4$-genera of symmetric knots.
For a periodic knot or a strongly invertible knot $K$,
one can define the {\it equivariant $4$-genus}\index{genus!equivariant $4$-genus} %%
$\tilde g_4(K)$
to be the minimum of the genera of smooth surfaces 
in $B^4$ bounded by $K$ that are invariant by a periodic diffeomorphism of $B^4$
extending the periodic diffeomorphism realizing the symmetry of $K$.
(Here, the symbol expressing the symmetry is suppressed in the symbol $\tilde g_4(K)$.)
Of course, $\tilde g_4(K)$ is bounded below by the 
$4$-genus $g_4(K)$\index{genus!$4$-genus}, %%
and it is invariant by equivariant cobordism.

The equivarinat cobordism of periodic knots was 
studied by Naik \cite{Naik1997},
where she gave criteria
for a given periodic knot to be equivariantly slice,
in terms of the linking number and the homology
of the double branched covering.
By using the criteria, she presented examples of slice periodic knots which are not equivariantly slice.
(See Davis-Naik \cite{Davis-Naik} and Cha-Ko \cite{Cha-Ko}
for further development.)

The equivariant cobordism\index{equivarinat cobordism}\index{cobordism!equivarinat cobordism} 
of strongly invertible knots was studied by 
the third author \cite{Sakuma1986}, where
the notion of directed strongly invertible knots was introduced
so that the connected sum is well-defined, and it was observed that
the set $\tilde\cob$ of the directed strongly invertible knots
modulo equivariant cobordism form 
a group 
with respect to the connected sum. 
He then introduced a polynomial invariant $\eta:\tilde\cob\to \ZZ\langle t\rangle$
which is a group homomorphism,
and presented a slice strongly invertible knot that is not equivariantly slice.
We note that it was recently proved by 
Prisa \cite{Prisa} that the group $\tilde\cob$ is not abelian.

These two old works show that there is a gap between the (usual) $4$-genus $g_4(K)$
and the equivariant $4$-genus $\tilde g_4(K)$
for both periodic knots and strongly invertible knots.

In a recent series of works 
\cite{Boyle-Issa2021a, Boyle-Issa2021b, Alfieri-Boyle, Boyle-Musyt,
Boyle-Chen2022a, Boyle-Chen2022b}, 
Boyle and his coworkers 
launched a project to systematically study the
equivariant $4$-genera of symmetric knots.
It should be noted that they treat not only periodic/strongly-invertible knots,
but also freely periodic knots and strongly negative amphicheiral knots.
In \cite[Theorems 2 and 3]{Boyle-Issa2021a},
Boyle and Issa gave a lower bound of $\tilde g_4(K)$ 
of a periodic knot $K$
in terms of the signatures
of $K$ and its quotient knot,
and showed that the gap between $g_4(K)$ and $\tilde g_4(K)$ can be arbitrary large.
They also introduced the notion of the butterfly $4$-genus $\widetilde{bg}_4(K)$
of a directed (or marked) strongly invertible knot $K$,
gave an estimate of $\widetilde{bg}_4(K)$ in terms of the $g$-signature,
and showed that the gap between $g_4(K)$ and $\widetilde{bg}_4(K)$ can be arbitrary large
\cite[Theorem 4]{Boyle-Issa2021a}.

In \cite{Dai-Mallick-Stoffregen},
Dai, Mallick and Stoffregen
introduced equivariant concordance invariants of strongly invertible knots
using knot Floer homology,
and showed that 
the gap between the equivariant $4$-genus $\tilde g_4(K)$ and the genus $g(K)$
for marked strongly invertible knots can be arbitrary large,
proving the conjecture \cite[Question 1.1]{Boyle-Issa2021a} by Boyle and Issa.
In fact, they constructed a homomorphism from $\tilde\cob$
to a certain group $\mathfrak{K}$, which is not a priori abelian.
By using the invariant, 
they proved that, for
the knot $K_n$ defined as the connected sum
$
K_n:=(T_{2n,2n+1}\# T_{2n,2n+1})\# -(T_{2n,2n+1}\# T_{2n,2n+1})
$,
the equivariant $4$-genus $\tilde g_4(K_n)$, 
with respect to some strong inversion,
is at least $2n-2$ \cite[Theorem 1.4]{Dai-Mallick-Stoffregen},
whereas $g_4(K_n)=0$.
Here $T_{2n,2n+1}$ is the torus knot of type $(2n,2n+1)$,
and \lq\lq$-$'' denoted the reversed mirror image.

On the other hand,
since $K_n$ is fibered,
the ($3$-dimensional) equivariant genera of 
the marked strongly invertible knots associated with $K_n$ 
are equal to the genus $g(K_n)=4n(2n-1)$
(see Proposition \ref{prop:fibered}),
which is bigger than the estimate $2n-2$ of $\tilde g_4(K_n)$ from below. 
As far as the authors know, there is no known algebraic invariant
that gives an effective estimate of the ($3$-dimensional) equivariant genera
of marked strongly invertible knots,
though the ($3$-dimensional) genus of any knot is estimated by the the Alexander polynomial 
and moreover it is determined by the Heegaard Floer homology 
(Ozsvath-Szabo \cite[Theorem 1.2]{Ozsvath-Szabo}).
As is noted in \cite{Dai-Mallick-Stoffregen},
there has been a renewed interest in strongly invertible knots
from the view point of more modern invariants
(see Watson \cite{Watson} and Lobb-Watson \cite{Lobb-Watson}).
Thus we would like to pose the following question.

\begin{question}
\label{question1}
{\rm
Is there a computable algebraic invariant of a marked strongly invertible knot 
that gives an effective estimate of the equivariant genus, 
or more strongly, determines it? 
In other words, is there a computable,
algebraically defined,
integer-valued function
$I$:
$(K,h,\delta)\mapsto I(K,h,\delta)$
on the set of all marked strongly invertible knots up to equaivalence,
such that $g(K,h,\delta) \ge I(K,h,\delta)$
and that it does not descend to a function on $\tilde\cob$?
}
\end{question}

The last requirement means that we have a chance to have a nontrivial estimate of
$g(K,h,\delta)$ by using $I$ even if the equivariant $4$-genus 
is strictly smaller than the equivariant genus.
If we drop this requirement, then 
such invariants are constructed 
in the above mentioned work by 
Dai, Mallick and Stoffregen \cite{Dai-Mallick-Stoffregen}.
If we restrict to the marked strongly invertible knots $(K,h,\delta)$
such that 
the constituent knot $\check K=k\cup \check \delta$ of
the quotient $\theta$-curve $\theta(K,h)$ is alternating,
then 
the works by Kalfagianni and Lee \cite{Kalfagianni-Lee},
Ito and Takimura \cite{Ito-Takimura2020a, Ito-Takimura2020b}
and {Kindred} \cite{Kindred}
on crosscap numbers of alternating links
give such invariants.
In fact, their results together with the classical results
by Murasugi \cite{Murasugi} and Crowell \cite{Crowell}
on the genera of alternating links
estimate/determine the band number 
$b(\check K)=\min(2g(\check K),\gamma(\check K))$ 
of the alternating knot $\check K$
in terms of the Jones polynomial of $\check K$.
Since $g(K,h,\delta)\ge b(\check K)$ by Corollary \ref{cor:band-number},
the function $I$ defined by
$I(K,h,\delta):=b(\check K)$
on the set of 
marked strongly invertible knots $(K,h,\delta)$
with alternating $\check K$ 
satisfies the desired property.
We note that all of the above mentioned works 
\cite{Kalfagianni-Lee, Ito-Takimura2020a, Ito-Takimura2020b, Kindred}
depend on the geometric study due to 
Adams and Kindred \cite{Adams-Kindred}
which gives an algorithm for determining the cross cap numbers
of prime alternating links.
We also note Burton and Ozlen \cite{Burton-Ozlen} 
present an algorithm that utilizes normal surface theory and integer programming 
to determine the crosscap numbers of generic knots.
However, as far as we know, there are no computable algebraic invariants
which give lower bounds of crosscap numbers of generic knots.
Moreover, what we really need to estimate is 
the smallest Betti number of the spanning surfaces for $\check K$ 
that satisfy the conditions in Proposition \ref{prop:band-number}
(cf. Proposition \ref{prop:lifting}).

{\bf Acknowledgements}\ 
M.H. and M.S learned the theory of divides
through a lecture by A'Campo in Osaka in 1999,
and R.H. learned it 
through a lecture by A'Campo's former student
Masaharu Ishikawa in Hiroshima in 2015.
The theory of divides and discussions with A'Campo 
have been sources of inspiration for the authors.
The authors would like to thank Professor Norbert A'Campo
for invaluable discussions and encouragements.
They would also like to thank the anonymous referee
for his or her valuable comments and suggestions 
that helped them to improve the exposition. 

M.H. is supported by JSPS KAKENHI JP18K03296.
M.S. is supported by JSPS KAKENHI JP20K03614
and by Osaka Central Advanced Mathematical Institute 
(MEXT Joint Usage/Research Center on Mathematics and Theoretical Physics JPMXP0619217849).

%
%
%Department of Mathematics,
%Nagoya Institute of Technology,\\
%Showa-ku, Nagoya city, Aichi, 466-8555, Japan\\
%\email{hirasawa.mikami@nitech.ac.jp}\\
%
%Inuyama-Minami High School,\\
%Hasuike 2-21, Inuyama City, Aichi, 484-0835, Japan\\
%\email{ryhiura@gmail.com}\\
%
%Osaka Central Advanced Mathematical Institute,
%Osaka Metropolitan University\\
%3-3-138, Sugimoto, Sumiyoshi, Osaka City
%558-8585, Japan\\
%\email{sakuma@hiroshima-u.ac.jp}
%

\end{document}